\documentclass[12pt,a4paper]{amsart}
\usepackage{amscd,amsmath,amsthm,amssymb}

\def\NZQ{\mathbb}               

\def\ZZ{{\NZQ Z}}
\def\RR{{\NZQ R}}

\def\Pc{{\mathcal P}}

\def\ab{{\mathbf a}}

\def\opn#1#2{\def#1{\operatorname{#2}}} 
\opn\chara{char} \opn\length{\ell} \opn\pd{pd} \opn\rk{rk}
\opn\projdim{proj\,dim} \opn\injdim{inj\,dim} \opn\rank{rank}
\opn\depth{depth} \opn\grade{grade} \opn\height{height}
\opn\embdim{emb\,dim} \opn\codim{codim}
\opn\Cl{Cl}

\opn\Tr{Tr} \opn\bigrank{big\,rank}
\opn\superheight{superheight}\opn\lcm{lcm}
\opn\trdeg{tr\,deg}
\opn\rdeg{rdeg}
	\opn\reg{reg} \opn\lreg{lreg} \opn\ini{in} \opn\lpd{lpd}
	\opn\size{size} \opn\sdepth{sdepth}
	\opn\link{link}\opn\fdepth{fdepth}\opn\lex{lex}
	\opn\tr{tr}
	\opn\type{type}
	\opn\gap{gap}
	\opn\arithdeg{arith-deg}
	\opn\revlex{revlex}
	\opn\div{div} \opn\Div{Div} \opn\cl{cl} \opn\Cl{Cl}
	\opn\Spec{Spec} \opn\Supp{Supp} \opn\supp{supp} \opn\Sing{Sing}
	\opn\Ass{Ass} \opn\Min{Min}\opn\Mon{Mon}
	\opn\Ann{Ann} \opn\Rad{Rad} \opn\Soc{Soc}
	\opn\Im{Im} \opn\Ker{Ker} \opn\Coker{Coker} \opn\Am{Am}
	\opn\Hom{Hom} \opn\Tor{Tor} \opn\Ext{Ext} \opn\End{End}
	\opn\Aut{Aut} \opn\id{id}
	
	\opn\nat{nat}
	\opn\pff{pf}
	\opn\Pf{Pf} \opn\GL{GL} \opn\SL{SL} \opn\mod{mod} \opn\ord{ord}
	\opn\Gin{Gin} \opn\Hilb{Hilb}\opn\sort{sort}
	\opn\PF{PF}\opn\Ap{Ap}
	\opn\mult{mult}
	\opn\bight{bight}
	\opn\div{div}
	\opn\Div{Div}
	\opn\aff{aff}
	\opn\relint{relint} \opn\st{st}
	\opn\lk{lk} \opn\cn{cn} \opn\core{core} \opn\vol{vol}  \opn\inp{inp} \opn\nilpot{nilpot}
	\opn\link{link} \opn\star{star}\opn\lex{lex}\opn\set{set}
	\opn\width{wd}
	\opn\Fr{F}
	\opn\QF{QF}
	\opn\G{G}
	\opn\type{type}\opn\res{res}
	\opn\conv{conv}
	\opn\Int{Int}
	\opn\Deg{Deg}
	\opn\Sym{Sym}
	\opn\Con{Con}
	\opn\gr{gr}
	
	\def\pot#1#2{#1[\kern-0.28ex[#2]\kern-0.28ex]}

	\opn\dirlim{\underrightarrow{\lim}}
	\opn\inivlim{\underleftarrow{\lim}}
	\def\Implies{\ifmmode\Longrightarrow \else
		\unskip${}\Longrightarrow{}$\ignorespaces\fi}
	\def\implies{\ifmmode\Rightarrow \else
		\unskip${}\Rightarrow{}$\ignorespaces\fi}
	\def\iff{\ifmmode\Longleftrightarrow \else
		\unskip${}\Longleftrightarrow{}$\ignorespaces\fi}

	\let\:=\colon
	\newtheorem{Theorem}{Theorem}[section]
	\newtheorem{Lemma}[Theorem]{Lemma}
	\newtheorem{Corollary}[Theorem]{Corollary}
	\newtheorem{Proposition}[Theorem]{Proposition}
	
	\theoremstyle{definition}
	\newtheorem{Remark}[Theorem]{Remark}
	
	\newtheorem{Example}[Theorem]{Example}
	\let\epsilon\varepsilon
	\let\kappa=\varkappa
	\opn\dis{dis}
	\def\pnt{{\raise0.5mm\hbox{\large\bf.}}}
	
	\opn\Lex{Lex}

\begin{document}

\title{A New Invariant of Lattice polytopes}
\author{Winfried Bruns}
\address{Institut f\"ur Mathematik, Universit\"at Osnabr\"uck, Albrechtstr. 28a, 49074 Osnabr\"uck, Germany}
\email{wbruns@uos.de}
\author{Takayuki Hibi}
\address{Takayuki Hibi, Department of Pure and Applied Mathematics, Graduate School of Information Science and Technology, Osaka University, Suita, Osaka 565-0871, Japan}
\email{hibi@math.sci.osaka-u.ac.jp}
\dedicatory{}
\begin{abstract}
The maximal degree of monomials belonging to the unique minimal system of monomial generators of the canonical module $\omega(K[\Pc])$ of the toric ring $K[\Pc]$  defined by a lattice polytope $\Pc$ will be studied.  It is shown that if $\Pc$ possesses an interior lattice point, then the maximal degree is at most $\dim \Pc - 1$, and that this bound is the best possible  in general.    
\end{abstract}	
\maketitle
\thispagestyle{empty}

\section*{Introduction}
The original motivation of the present paper is to investigate the maximal degree of monomials belonging to the unique minimal system of monomial generators of the canonical module of the toric ring defined by a lattice polytope.

Let $T = K[x_1, x_1^{-1}, \ldots, x_m, x_m^{-1}, t, t^{-1}]$ denote the Laurent polynomial ring in $m + 1$ variables over a field $K$.  One associates each $\ab = (a_1, \ldots, a_m) \in \ZZ^m$ with the Laurent monomial ${\bf x}^\ab = x_1^{a_1} \cdots x_d^{a_m} \in T$.  Let $\Pc \subset \RR^m$ be a lattice polytope of dimension $d$ (we do not exclude $d < m$). One naturally identifies $\Pc$ with $\Pc \times \{1\} \subset \RR^{m+1}$.  Let $C(\Pc) \subset \RR^{m+1}$ denote the rational polyhedral cone spanned by $\Pc \times \{1\}$. The {\em toric ring} of $\Pc$ is the subring
\[
K[\Pc] = K[{\bf x}^\ab t^n : (\ab,n) \in C(\Pc) \cap \ZZ^{m+1}]
\] 
of $T$. The algebraic and combinatorial study of toric rings is contained in several books, for example in \cite{BG} and \cite{Hibi_red_book}.  We set $\deg({\bf x}^\ab t^n) = n$ for all $\ab\in\ZZ^m$.  Then $K[\Pc]$ is a positively graded domain. It is not necessarily generated by degree $1$ elements; if it is so, one says that $\Pc$ is integrally closed in $\ZZ^m$ or, in other terminology, satisfies the integer decomposition property (IDP). (For the normality of $\Pc$ one must replace $\ZZ^m$ as the lattice of reference by the sublattice generated by the lattice points in $\Pc$.)
 
The toric ring $K[\Pc]$ is a normal semigroup ring, and therefore Cohen--Macaulay by Hochster's theorem. By the theorem of Danilov--Stanley its canonical module $\omega(K[\Pc])$ is generated by those ${\bf x}^\ab t^n$ for which $(\ab,n)$ is in the (relative) interior of $C(\Pc)$.  Let $G(\omega(K[\Pc]))$ denote the unique minimal system of monomial generators of $\omega(K[\Pc])$.  We study the maximal degree of monomials belonging to $G(\omega(K[\Pc]))$.  Among several results, especially Corollary \ref{canonical_module} says that, if a lattice polytope of dimension $d\ge 2$ possesses an interior lattice point, then the maximal degree of monomials in $G(\omega(K[\Pc]))$ is at most $d - 1$.  Furthermore, the upper bound is best possible (Corollary \ref{best_possible}).  

{\em Normaliz} \cite{Normaliz} was helpful to the preparation of the present paper. The methods used below are inspired  by those in \cite[Section 1.2]{BG}.

Instead of $\ZZ^m$ one can start with a free abelian group $M$ as is common in toric geometry.  However, since our target is convex polytopes themselves, in the present paper, we work in the frame of  $\ZZ^m$.       

\section{Reduced $\Pc$-degree}
In the following we will often switch from a polytope $\Pc \subset \RR^m$ to $\Pc \times \{1\} \subset \RR^{m+1}$ without changing the name, and we may also identify monomials and lattice points.

Recall that a convex polytope $\Pc \subset \RR^m$ of dimension $d$ is a {\em lattice polytope} if each vertex of $\Pc$ belongs to $\ZZ^m$.  Let
\[
M(\Pc) =  \sum_{x \,\in\, (\Pc \times \{1\}) \,\cap\, \ZZ^{m+1}} \ZZ_+ \, x \subset C(\Pc)
\]
be the semigroup generated by $(\Pc \times \{1\}) \cap \ZZ^{m+1}$. Furthermore, let $\Int C(\Pc)$ denote the interior of $C(\Pc)$ and $M^*(\Pc) = \Int C(\Pc) \cap \ZZ^{m+1}$.  As already said, the monomials associated with the lattice points in $M^*(\Pc)$ span the canonical module of $K[\Pc]$.

In accordance with the degree of monomials introduced above, the degree of $y = (y_1, \ldots, y_d, y_{m+1})$ belonging to $C(\Pc) \cap \ZZ^{m+1}$ is $\deg y = y_{m+1}$.  Each $y \in M^*(\Pc)$ can be expressed as $y = z + w$ with $z \in M^*(\Pc)$ and $w \in M(\Pc)$ in such a way that $z$ has minimal degree for all possible choices of $w$.  The {\em reduced $\Pc$-degree} of $y$ is 
\[
\rdeg y = \rdeg_\Pc y = \deg z.
\]
One says that $y \in M^*(\Pc)$ is {\em $M(\Pc)$-irreducible} if $\rdeg y = \deg y$.  Clearly the set $G(\Pc)$ of all $M(\Pc)$-irreducible elements is the unique minimal system of generators of $M^*(\Pc)$ with respect to the action of $M(\Pc)$ by addition.  
 
In the present paper, the invariant, called the {\em  int\hbox{$^*$}degree}, 
\begin{equation}
\label{reduced_degree}
\max\{ \rdeg y : y \in M^*(\Pc)\}
\end{equation}
of $\Pc$ is studied.  As indicated above, a bound for this invariant implies the same bound for the minimal set of generators of the canonical module of $K[\Pc]$.

\section{Lattice simplices}
First, the int\hbox{$^*$}degree (\ref{reduced_degree}) of a lattice simplex is studied.     

\begin{Lemma}
\label{simplex_a}
Let $\sigma \subset \RR^m$ be a lattice simplex of dimension $d$.  Then $$\max\{ \rdeg y : y \in M^*(\sigma)\} \leq d+1.$$
\end{Lemma} 

\begin{proof}
Let $x_0, x_1, \ldots, x_d$ be the vertices of $\Pc$.  Each $y \in M^*(\Pc)$ is expressed uniquely as $y = \sum_{i=0}^d q_i x_i$ with each $q_i > 0$.  Let $q_i' = q_i - \lceil  q_i - 1 \rceil$.  Then $0 < q_i \leq 1$.  One has $y = z + w$, where
\[
z = \sum_{i=0}^d q'_i x_i,\qquad 
w = \sum \lceil  q_i - 1 \rceil x_i.
\]
Since $\deg z \leq d + 1$, one has $\rdeg y \leq d + 1$, as desired. 
\end{proof}

A lattice simplex $\sigma \subset \RR^m$ is said to be {\em empty} if $\sigma \cap \ZZ^m$ coincides with the set of vertices of $\sigma$. Nontrivial empty $d$-simplices exist for all $d\ge 3$; see \cite[Remark 2.55]{BGT}.

\begin{Lemma}
\label{simplex_b}
Let $\sigma \subset \RR^m$ be a lattice simplex of dimension $d$.  If $\sigma$ is nonempty, then 
\[
\max\{ \rdeg y : y \in M^*(\sigma)\} \leq d.
\]
\end{Lemma}

\begin{proof}
Let $x_0, x_1, \ldots, x_d$ be the vertices of $\sigma$.  The proof of Lemma \ref{simplex_a} says that $y \in M^*(\sigma)$ is $M(\sigma)$-irreducible if and only if
\begin{equation}
\label{irreducible}
y = \sum_{i=0}^d q_i x_i,\qquad 0 < q_0, q_1, \ldots, q_d \leq 1.
\end{equation}
Our work is then to show that $y$ cannot be $M(\sigma)$-irreducible if $\deg y = d + 1$.  Obviously, one has $y = x_0 + \cdots + x_d$ if $\deg y = d+1$.  Since $\sigma$ is nonempty, there exists $x \in \sigma \cap \ZZ^{m+1}$ with $x = \sum_{i=0}^{d} r_i x_i$ with each $0 \leq r_i < 1$ and $\sum_{i=0}^{d} r_i = 1$.  Since  $0 \leq r_i < 1$ for all $i$, one has $y - x \in M^*(\sigma)$.  Thus $y$ cannot be $M(\sigma)$-irreducible. 
\end{proof} 

\begin{Lemma}
\label{simplex_c}
Let $\sigma \subset \RR^m$ be a lattice simplex of dimension $d$ and $x_0, x_1, \ldots, x_d$ the vertices of $\sigma$.  If $\sigma$ is empty, then 
\[
\rdeg(x_0 + \cdots + x_d) = d + 1
\]
and, for all $y \in M^*(\sigma)$, one has
\[
\rdeg y \neq d.
\]
\end{Lemma}

\begin{proof}
Since $\sigma$ is empty, 
\[
M(\sigma) = \ZZ_+ x_0 + \cdots + \ZZ_+ x_d
\]
and $x_0, \ldots, x_d$ are linearly independent.  Thus $(x_0 + \cdots + x_d) - w \not\in M^*(\sigma)$ for all nonzero $ w \in M(\sigma)$, and $x_0 + \cdots + x_d$ is $M(\sigma)$-irreducible.  

Now, suppose that there exists $y \in M^*(\sigma)$ with $\rdeg y = d$ which is $M(\sigma)$-irreducible.  Then $y$ is of the form (\ref{irreducible}), $(x_0 + \cdots + x_d) - y$ is of degree $1$ and belongs to $C(\sigma) \cap \ZZ^{m+1}$.  Hence $(x_0 + \cdots + x_d) - y$ must be a vertex of $\sigma$, say, $x_0$.  Thus $y = x_1 + \cdots + x_d$, which cannot belong to $M^*(\sigma)$. 
\end{proof}

\section{Lattice polytopes}
Now we study the int\hbox{$^*$}degree (\ref{reduced_degree}) of a general lattice polytope.

\begin{Theorem}
\label{winfried_a}
Let $\Pc \subset \RR^m$ be a lattice polytope of dimension $d$.  
\begin{itemize}
\item[(a)]  One has $\rdeg y \leq d + 1$ for all $y \in M^*(\Pc)$.

\item[(b)] The following conditions are equivalent:
\begin{itemize}
\item[(i)] $\Pc$ is not an empty simplex;

\item[(ii)] $\rdeg y \leq d$ for all $y \in M^*(\Pc)$.
\end{itemize}
\end{itemize}
\end{Theorem}

\begin{proof}
Let $\Sigma$ be a lattice triangulation of $\Pc$.  In other words, $\Sigma$ is a triangulation of $\Pc$ for which each $\sigma\in\Sigma$ is an empty simplex, and consequently, each $x \in \Pc \cap \ZZ^m$ is a vertex of some $\sigma \in \Sigma$.  Let $\Sigma'$ denote the subset of $\Sigma$ consisting of those $\sigma \in \Sigma$ for which $(\sigma \setminus \partial \sigma) \cap \partial \Pc = \emptyset$. ($\partial$ denotes the relative boundary.)

Since $\Pc$ is the disjoint union of $\sigma \setminus \partial \sigma$ with $\sigma \in \Sigma$, it follows that $\Int C(\Pc)$ is the disjoint union of $\Int C(\sigma)$ with $\sigma \in \Sigma'$.  Thus {\rm (a)} follows from Lemma \ref{simplex_a}.  On the other hand, {\rm (ii)} $\Rightarrow$ {\rm (i)} in {\rm (b)} follows from Lemma \ref{simplex_c}. 

 Now, in order to prove {\rm (i)} $\Rightarrow$ {\rm (ii)} in {\rm (b)}, suppose that $y \in M^*(\Pc)$ belongs to $\Int C(\sigma)$ with $\sigma \in \Sigma'$.  One may assume that $\dim \sigma = d$.  Since $\rdeg_\Pc y \leq \rdeg_\sigma y \leq d + 1$, it is enough to exclude $\rdeg_\sigma y = d + 1$.  Let $\rdeg_\sigma y = d + 1$.   Then $y = x_0 + x_1 + \cdots + x_d$, where $x_0, x_1, \ldots, x_d$ are the vertices of $\sigma$.  Since $\Pc$ is not an empty simplex, one has $\Pc \neq \sigma$, which guarantees the existence of a facet $\tau$ of $\sigma$ with $\tau \setminus \partial \tau \subset \Pc \setminus \partial \Pc$.  Let, say, $x_0 \not\in \tau$.  Since $(\tau \setminus \partial \tau) \cap \partial \Pc = \emptyset$, it follows that $x_1 + \cdots + x_d \in {\rm Int}C(\Pc)$.  Thus $y = x_0 + (x_1 + \cdots + x_d)$ and $\rdeg y \leq d$, a contradiction.
\end{proof}

When $(\Pc \setminus \partial \Pc) \cap \ZZ^m \neq \emptyset$, the invariant (\ref{reduced_degree}) can be improved.  

\begin{Theorem}
\label{winfried_b}
Let $\Pc \subset \RR^m$ be a lattice polytope of dimension $d \geq 2$ and suppose that $(\Pc \setminus \partial \Pc) \cap \ZZ^m \neq \emptyset$.  Then $\rdeg y \leq d - 1$ for all $y \in M^*(\Pc)$. 
\end{Theorem}

In order to prove Theorem \ref{winfried_b}, a special lattice triangulation using interior lattice points effectively is required.

\begin{Lemma}
\label{triangulation}
Let $\Pc$ be a lattice polytope with an interior lattice point.  Then $\Pc$ has a lattice triangulation $\Sigma$ for which $\sigma\subset \partial \Pc$  if all vertices of $\sigma$ are in $\partial \Pc$.
\end{Lemma}

\begin{proof}
Let $\Delta$ be a lattice triangulation of $\partial \Pc$ and $x$ an interior point of $\Pc$.  Then one defines the triangulation $\Sigma'$ of $\Pc$ by $\Delta\cup\{ \conv(\sigma, x): \sigma \in \Delta\}$.  Finally $\Sigma$ is obtained from $\Sigma'$ by stellar subdivision with respect to the remaining interior points of $\Pc$ in an arbitrary order.
\end{proof}  

\begin{proof}[Proof of Theorem \ref{winfried_b}]
Let $\Sigma$ be a lattice triangulation constructed in Lemma \ref{triangulation}.  In particular, every $\sigma \in \Sigma$ with $(\sigma \setminus \partial \sigma) \cap \partial \Pc = \emptyset$ possesses a vertex belonging to $(\Pc \setminus \partial \Pc) \cap \ZZ^m$.  Let $y \in M^*(\Pc)$ belong to $\sigma \setminus \partial \sigma$ with $\sigma \in \Sigma$.  Let $y = \sum_{i=0}^{q} r_i x_i$, where $x_0, \ldots, x_q$ are the vertices of $\sigma$, where $x_0 \in (\Pc \setminus \partial \Pc) \cap \ZZ^m$ and where $q \leq d$.  Let $s_i = r_i - \lceil r_i - 1 \rceil$.  Then $0 < s_i \leq 1$ for $i = 0,1,\ldots,q$ and  
\[
y = \sum_{i=0}^{q} s_i x_i + \sum_{i=0}^{q} \lceil r_i - 1 \rceil x_i.
\]
Assume first that  $s_0 < 1$ and $y_0 = \sum_{i=0}^{q} s_i x_i \in M^*(\Pc)$.  Then $\deg y_0 = \sum_{i=0}^{q} s_i \leq q \leq d$.  Thus, by using Lemma \ref{simplex_c}, one has $\rdeg y_0 < d$.  Finally, if $s_0 = 1$, then $r_0$ is a positive integer and 
\[
y = x_0 + \left[(r_0 - 1) x_0 + \sum_{i=1}^{q} r_i x_i\right].
\]  
Thus $\rdeg y = 1$.   
\end{proof}

\begin{Corollary}
\label{canonical_module}
Let $\Pc$ be a lattice polytope of dimension $d$ and $K[\Pc]$ its toric ring. Then the maximal degree of monomials belonging to the unique minimal system of monomial generators of the canonical module of $K[\Pc]$ is bounded by $d$.

If $\Pc$ possesses an interior lattice point, it is bounded by $d-1$.
\end{Corollary}

\begin{Remark}\label{deg 1 subalgbera}
The bound in Corollary \ref{canonical_module} is the degree bound for the minimal system of generators of the canonical module with respect to the subalgebra generated by the degree $1$ elements of the toric ring. In general the minimal system of generators with respect to the full toric ring is smaller, as Proposition \ref{empty_toric} below shows, but in general it is difficult to get information on it. In the next section we will consider examples in which the toric ring is generated in degree $1$ and the bound in Corollary \ref{canonical_module} is optimal.
\end{Remark}

Recall that a lattice $d$-simplex with vertices $v_0,\dots,v_d\in \RR^m$ is \emph{unimodular} if its $d$-dimensional volume is $1/d!$. An equivalent condition is that the lattice vectors $(v_0,1),\dots,(v_d,1)$ generate a direct summand of $\ZZ^{m+1}$, or in algebraic terms, that the toric ring is isomorphic to a polynomial ring in $d+1$ variables, generated by the monomials ${\bf x}^{(v_0,1)},\dots, {\bf x}^{(v_d,1)}$ as a $K$-algebra.
In particular a unimodular simplex is empty.

\begin{Proposition} \label{empty_toric}
Let $\sigma$ be an empty lattice $d$-simplex in $\RR^m$. 
\begin{itemize}
\item[(a)] If $\sigma$ is unimodular, then the canonical module of $K[\sigma]$ is generated by a single element of degree $d+1$ as a $K[\Pc]$-module.

\item[(b)] If $\sigma$ is non-unimodular, then the canonical module of $K[\sigma]$ is generated by elements of degree $\le d-1$ as a $K[\Pc]$-module.
\end{itemize}
\end{Proposition}

\begin{proof}
For (a) it is enough to observe that every lattice point in $\Int C(\Pc)$ has a unique representation as a linear combination of $(v_0,1),\dots,(v_d,1)$ with nonnegative coefficients in $\ZZ$.

For (b) we use that there must exist a nonzero $w=q_0(v_0,1)+\dots+q_d(v_d,1)\in C(\Pc)\cap \ZZ^{m+1}$ with $0\le q_i < 1$ for $i=0,\dots,d$. After our previous results we must only exclude that  $y=(v_0,1)+\dots+(v_d,1)$ belongs to the minimal system of generators. But this is clear, since $y-w\in\Int C(\Pc)$.
\end{proof}

\section{Examples}
It is easy to find a $d$-simplex $\sigma$ with $\rdeg x = d$ for all $x\in M^*(\sigma)$.

\begin{Example}
Define  $\sigma$ by its vertices
$0, 2e_1,e_2\dots,e_d$, where $e_i$ is the $i$th unit vector in $\RR^d$. Then $\Int C(\sigma)$ is generated by $(1,\dots,1,d) \in \RR^{d+1}$ with respect to the action of $M(\sigma)$ by addition. Moreover the toric ring $K[\sigma]$ is generated by the monomials ${\bf x}^\ab t$ where $\ab \in M(\sigma)$.
\end{Example}

Now we give an example that shows the optimality of Theorem \ref{winfried_b}. 

\begin{Example}
Let $\Pc \subset \RR^d$, $d\ge 2$, denote the polytope defined by the system of inequalities
\begin{align*}
& 0 \leq x_i \leq 2, \qquad 1 \leq i \leq d - 1, \\
& 0 \leq x_d \leq d, \\
& x_1 + \cdots + x_d \leq d + 1.
\end{align*}
It is known \cite[Example 2.6 (c)]{HH} that $\Pc$ is a lattice polytope of dimension $d\ge 2$ whose toric ring is generated in degree $1$.  One has $(1, \ldots, 1) \in \Pc \setminus \partial \Pc$.  Let 
\[
y_i = (1, \ldots, 1, (i-1)d + i, i) \in \ZZ^{d+1}, \qquad  i = 1,2,\ldots,d-1.
\]
Then $y_i \in M^*(\Pc)$ with $\deg y_i = i$.
\end{Example}

\begin{Lemma}
Each $y_i$ is $M(\Pc)$-irreducible.
\end{Lemma}

\begin{proof}
Let $y_i = z + w$ with $z \in M^*(\Pc)$ and $w \in M(\Pc)$.  Let
\[
z = (1, \ldots, 1, a, j), \qquad w = (0, \ldots, 0, b, k).
\]
Then
\[
(d - 1) + a < j (d + 1), \qquad b \leq kd, \qquad i = j + k.
\]
Thus
\[
(i-1)d + i = a + b < j (d + 1) - (d - 1) + kd = (i - 1)d + (j + 1).
\]
Hence $i < j + 1$, a contradiction. 
\end{proof}

\begin{Corollary}
\label{best_possible}
One has $\{\rdeg y : y \in M^*(\Pc) \} = \{1,2, \ldots,d-1\}$.
\end{Corollary}

The above observation shows that Theorem \ref{winfried_b} cannot be improved even if one replaces the action of $M(\Pc)$ by that of the semigroup $C(\Pc)\cap \ZZ^{m+1}$: the two semigroups coincide for $\Pc$. 

{}

\end{document}